\documentclass{amsart}%
\usepackage{amsfonts}
\usepackage{amsmath}
\usepackage{amssymb}
\usepackage{graphicx}%
\providecommand{\U}[1]{\protect\rule{.1in}{.1in}}
\newtheorem{theorem}{Theorem}
\theoremstyle{plain}

\newtheorem{corollary}{Corollary}

\newtheorem{lemma}{Lemma}

\newtheorem{proposition}{Proposition}

\numberwithin{equation}{section}
\begin{document}
\title[ ]{Symmetric itinerary sets}
\author{Michael F. Barnsley}
\address{Department of Mathematics\\
Australian National University\\
Canberra, ACT, Australia\\
 \\
 }
\author{Nicolae Mihalache}
\address{KTH-Royal Institute of Technology,\\
Deparment of Mathematics,\\
100 44 Stockholm, Sweden\\
 }
\email{michael.barnsley@maths.anu.edu.au,}
\urladdr{http://www.superfractals.com{}}

\begin{abstract}
We consider a one parameter family of dynamical systems $W:[0,1]\rightarrow
\lbrack0,1]$ constructed from a pair of monotone increasing diffeomorphisms
$W_{i}$, such that $W_{i}^{-1}:$ $[0,1]\rightarrow\lbrack0,1],$ $(i=0,1)$. We
characterise the set of symbolic itineraries of $W$ using an attractor
$\overline{\Omega}$ of an iterated closed relation, in the terminology of
McGehee, and prove that there is a member of the family for which
$\overline{\Omega}$ is symmetrical.

\end{abstract}
\maketitle

\section{Introduction}

Let $W_{0}:[0,a]\rightarrow\lbrack0,1]$ and $W_{1}:[1-b,1]\rightarrow
\lbrack0,1]$ be continuous and differentiable, and such that $a+b>1$,
$W_{0}(0)=W_{1}(1-b)=0,$ $W_{0}(a)=W_{1}(1)=1.$ Let the derivatives
$W_{i}^{\prime}(x)$ $(i=0,1)$ be uniformly bounded below by $d>1$.%

\begin{figure}[ptb]%
\centering
\includegraphics[
height=3.4428in,
width=3.55in
]%
{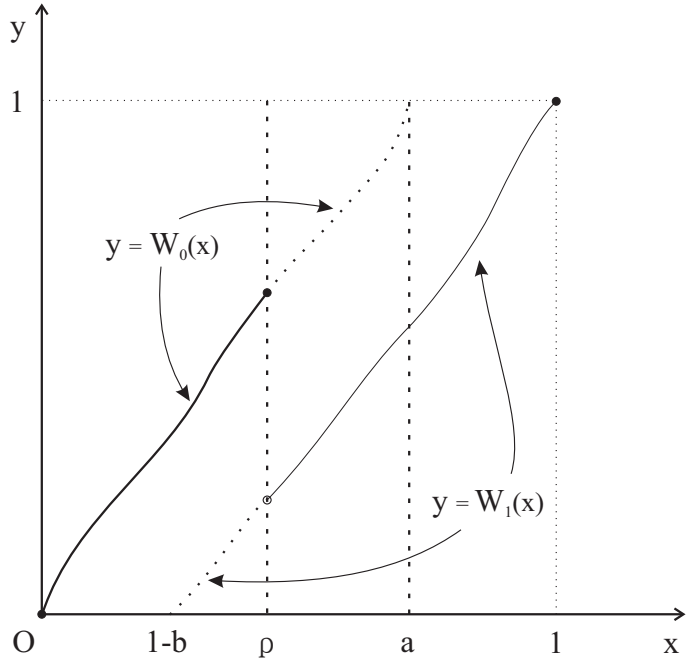}%
\caption{The piecewise continuous dynamical system W:[0,1]$\rightarrow$[0,1]
is defined in terms of two monotone strictly increasing differentiable
functions W$_{\text{0}}$(x) and W$_{\text{1}}$(x), and a real parameter $\rho
$.}%
\label{Figure1}%
\end{figure}

For $\rho\in\lbrack1-b,a]$ we define $W:[0,1]\rightarrow\lbrack0,1]$ by
\[
\lbrack0,1]\ni x\mapsto\left\{
\begin{array}
[c]{c}%
W_{0}(x)\text{ if }x\in\lbrack0,\rho]\\
W_{1}(x)\text{ otherwise}.
\end{array}
\right.
\]
See Figure \ref{Figure1}. Similarly, we define $W_{+}:[0,1]\rightarrow
\lbrack0,1]$ by replacing $[0,\rho]$ by $[0,\rho)$.

Let $I=\{0,1\}$. Let $I^{\infty}=\{0,1\}\times\{0,1\}\times\cdots$ have the
product topology induced from the discrete topology on $I$. For $\sigma\in
I^{\infty}$ write $\sigma=\sigma_{0}\sigma_{1}\sigma_{2}\ldots,$ where
$\sigma_{k}\in I$ for all $k\in\mathbb{N}$. The product topology on
$I^{\infty}$ is the same as the topology induced by the metric $d(\omega
,\sigma)=2^{-k}$ where $k$ is the least index such that $\omega_{k}\neq
\sigma_{k}$. It is well known that $(I^{\infty},d)$ is a compact metric space.
We define a total order relation $\preceq$ on $I^{\infty},$ and on $I^{n}$ for
any $n\in\mathbb{N}$, by $\sigma\prec\omega$ if $\sigma\neq\omega$ and
$\sigma_{k}<\omega_{k}$ where $k$ is the least index such that $\sigma_{k}%
\neq\omega_{k}$. For $\sigma\in I^{\infty}$ and $n\in\mathbb{N}$ we write
$\sigma|_{n}=\sigma_{0}\sigma_{1}\sigma_{2}...\sigma_{n}$. $I^{\infty}$ is the
appropriate space in which to embed and study the itineraries of the family of
discontinuous dynamical systems $W:[0,1]\rightarrow\lbrack0,1]$.

For $W_{(+)}\in\{W,W_{+}\}$ let $W_{(+)}^{k}$ denote $W_{(+)}$ composed with
itself $k$ times, for $k\in\mathbb{N}$, and let $W_{(+)}^{-k}=(W_{(+)}%
^{k})^{-1}$. We define a map $\tau:[0,1]\rightarrow I^{\infty}$, using all of
the orbits of $W,$ by
\[
\tau(x)=\sigma_{0}\sigma_{1}\sigma_{2}\ldots
\]
where $\sigma_{k}$ equals $0,$ or $1,$ according as $W^{k}(x)\in\lbrack
0,\rho],$ or $(\rho,1],$ respectively. We call $\tau(x)$ the
\textit{itinerary} of $x$ under $W$, or an \textit{address }of $x$, and we
call $\Omega=\tau([0,1])$ an \textit{address space} for $[0,1]$. Similarly, we
define $\tau^{+}:[0,1]\rightarrow I^{\infty}$ so that $\tau^{+}(x)_{k}$ equals
$0$, or $1,$ according as $W_{+}^{k}(x)\in\lbrack0,\rho),$ or $[\rho,1],$
respectively; and we define $\Omega_{+}=\tau^{+}([0,1])$. Note that $W$,
$W_{+}$, $\Omega,$ $\Omega_{+},$ $\tau$, and $\tau^{+}$ all depend on $\rho$.

The main goals of this paper are to characterise $\overline{\Omega}$ and to
show that there exists a value of $\rho$ such that $\overline{\Omega}$ is symmetric.

\begin{theorem}
\label{irtheorem} $\overline{\Omega}$ is the maximal attractor of the iterated
closed relation $r\subset I^{\infty}\times I^{\infty}$ defined by%
\[
r=\{(0\sigma,\sigma)\in I^{\infty}\times I^{\infty}:\sigma\preceq\alpha
\}\cup\{(1\sigma,\sigma)\in I^{\infty}\times I^{\infty}:\beta\preceq\sigma\}
\]
where $\alpha=\tau(W_{0}(\rho))$ and $\beta=\tau^{+}(W_{1}(\rho))$. The chain
recurrent set of $r$ is $\{\overline{0},\overline{1}\}\cup\{\sigma\in
\overline{\Omega}:\beta\preceq\sigma\preceq\alpha\}$ where $\{\sigma
\in\overline{\Omega}:\beta\preceq\sigma\preceq\alpha\}$ is a transitive
attractor for $r$, with basin of attraction $\overline{\Omega}\backslash
\{\overline{0},\overline{1}\}$ and dual repeller $\left(  I^{\infty}%
\backslash\overline{\Omega}\right)  \cup\{\overline{0},\overline{1}\}$.
\end{theorem}

We write $\overline{E}$ to denote the closure of a set $E$. But we write
$\overline{0}=000...,\overline{1}=111...\in I^{\infty}$. For $\sigma
=\sigma_{0}\sigma_{1}\sigma_{2}\ldots\in I^{\infty}$ we write $0\sigma$ to
mean $0\sigma_{0}\sigma_{1}\sigma_{2}\ldots\in I^{\infty}$ and $1\sigma
=0\sigma_{0}\sigma_{1}\sigma_{2}\ldots\in I^{\infty}$.

Define a symmetry function $^{\ast}:I^{\infty}\rightarrow I^{\infty}$ by
$\sigma^{\ast}=\omega$ where $\omega_{k}=1-\sigma_{k}$ for all $k$.

\begin{theorem}
\label{symmetrythm} There exists a unique $\rho\in\lbrack1-b,a]$ such that
$\overline{\Omega}^{\ast}=\overline{\Omega}$.
\end{theorem}

Theorem \ref{irtheorem} tells us that $\overline{\Omega}$ is fixed by
itineraries of two inverse images of the critical point $\rho$, and provides
the basis for a stable algorithm to determine $\overline{\Omega}$. It relates
the address spaces of dynamical systems of the form of $W$ to the beautiful
theory of iterated closed relations on compact Hausdorff spaces \cite{mcgehee}%
, and hence to the work of Charles Conley.

Theorem \ref{symmetrythm} is interesting in its own right and also because it
has applications in digital imaging, as explained and demonstrated, in the
special case of affine maps, in \cite{BHI}. It enables the construction of
parameterized families of nondifferentiable homeomorphisms on $[0,1]$, using
pairs of overlapping iterated function systems, see Proposition
\ref{proposition7}. Theorem \ref{symmetrythm} generalizes results in
\cite{BHI} to nonlinear $W_{i}$'s. The proof uses symbolic dynamics in place
of the geometrical construction outlined in \cite{BHI}. The approach and
results open up the mathematics underlying \cite{BHI} and \cite{igudesman}.

To tie the present work into \cite{BHI}, note that $\tau$ is a section, as
defined in \cite{BHI}, for the hyperbolic iterated function system
\[
\mathcal{F}:=([0,1];W_{0}^{-1},W_{1}^{-1})\text{.}%
\]
Our observations interrelate to, but are more specialized than, the work of
Parry \cite{parry}. Our point of view is topological rather than
measure-theoretic, and our main results appear to be new.

\section{\label{basicsec}Basic properties of $\tau$}

The following list of properties is relatively easy to check. Below the list
we elaborate on points 1, 2, and 3.

\begin{enumerate}
\item $W^{n}$ is piecewise differentiable and its derivative is uniformly
bounded below by $d^{n}$; each, except the leftmost branch of $W^{n}$, is
defined on an interval of the form $(r,s].$ $W_{+}^{n}$ is piecewise
differentiable and its derivative is uniformly bounded below by $d^{n}$; each,
except for the rightmost branch of $W_{+}^{n},$ is defined on an interval of
the form $[r,s).$

\item If $(r,s)$ is the interior of the definition domain of a branch of
$W^{n}$ (and of $W_{+}^{n})$ then $\tau(x)|_{n}$ is constant on $(r,s]$,
$\tau^{+}(x)|_{n}$ is constant on $[r,s)$, and $\tau(x)|_{n}=\tau^{+}(x)|_{n}$
for all $x\in(r,s)$.

\item The boundary of the definition domain of a branch of $W^{n}$ is
contained in $\{0,1\}\cup\bigcup\limits_{k=0}^{n-1}W^{-k}(\rho)$; by (1), the
length of such a domain is at most $d^{-n}$.

\item The set $\bigcup\limits_{k\in\mathbb{N}}W^{-k}(\rho)$ is dense in
$[0,1].$ This follows from (3).

\item $\tau(x)=\tau^{+}(x)$ unless $x\in\bigcup\limits_{k\in\mathbb{N}}%
W^{-k}(\rho)$ in which case $\tau(x)\prec\tau^{+}(x)$.

\item Both $\tau(x)$ and $\tau^{+}(x)$ are strictly increasing functions of
$x\in\lbrack0,1]$ and $\tau(x)\preceq\tau^{+}(x)$. This follows from (4) and (5).

\item For all $x\in\lbrack0,1],$ $\tau(x)$ is continuous from the left,
$\tau^{+}(x)$ is continuous from the right. Moreover, for all $x\in(0,1),$%
\[
\tau(x)=\lim_{\varepsilon\rightarrow0+}\tau^{+}(x-\varepsilon)\text{ and }%
\tau^{+}(x)=\lim_{\varepsilon\rightarrow0+}\tau(x+\varepsilon).
\]
These assertions follow from (2), (3) and (4).

\item Each $x\in W^{-n}(\rho),$ such that $\tau(x)|_{n}$ is constant, moves
continuously with respect to $\rho$ with positive velocity bounded above by
$d^{-n}.$ This follows from (1).

\item For $x\in(0,1)\backslash\bigcup\limits_{k=0}^{n}W^{-k}(\rho),$
$\tau(x)|_{n}=\tau^{+}(x)|_{n}$ is locally constant with respect to $\rho;$
moreover, this holds if $x$ depends continuously on $\rho$. This follows from
(2), (3) and (6).

\item The symmetry function $^{\ast}:I^{\infty}\rightarrow I^{\infty}$ is
strictly decreasing and continuous.

\item For any $\sigma|_{n}\in I^{n},$ $n\in\mathbb{N}$, the set
\[
\mathcal{I}(\sigma|_{n}):=\{x\in\lbrack0,1]:\tau(x)|_{n}=\sigma|_{n}\text{ or
}\tau^{+}(x)|_{n}=\sigma|_{n}\},
\]
is either empty or a non-degenerate compact interval of length at most
$d^{-n}.$ This follows from (2), (3) and (6).

\item The projection $\hat{\pi}:I^{\infty}\rightarrow\lbrack0,1]$ is
well-defined by%
\[
\hat{\pi}(\sigma)=\sup\{x\in\lbrack0,1]:\tau^{+}(x)\preceq\sigma\}=\inf
\{x\in\lbrack0,1]:\tau(x)\succeq\sigma\}\text{.}%
\]
This follows from (6).

\item The projection $\hat{\pi}:I^{\infty}\rightarrow\lbrack0,1]$ is
increasing, by (6); continuous, by (11); and, by (7),%
\begin{align*}
\hat{\pi}(\tau(x))  &  =\hat{\pi}(\tau^{+}(x))=x\text{ for all }x\in
\lbrack0,1],\\
\tau(\hat{\pi}(\sigma))  &  \preceq\sigma\preceq\tau^{+}(\hat{\pi}%
(\sigma))\text{ for all }\sigma\in I^{\infty}.
\end{align*}

\item Let $S:I^{\infty}\rightarrow I^{\infty}$ denote the left-shift map
$\sigma_{0}\sigma_{1}\sigma_{2}...\mapsto\sigma_{1}\sigma_{2}\sigma_{3}...$.
For all $\sigma\in I^{\infty}$ such that $\sigma\preceq\tau(\rho)$ or
$\sigma\geq\tau^{+}(\rho),$%
\[
\hat{\pi}(S(\sigma))=W(\hat{\pi}(\sigma))\text{.}%
\]
Also $\hat{\pi}(\tau^{+}(\rho))=\rho$ and $\hat{\pi}(S(\tau^{+}(\rho
)))=W_{1}(\rho)$. These statements follow from (7).
\end{enumerate}

Here we elaborate on points (1), (2) and (3). Consider the piecewise
continuous function $W^{k}(x)$, for $k\in\{1,2,...\}$. Its discontinuities are
at $\rho$ and, for $k>1$, other points in $(0,1)$, each of which can be
written in the form $W_{\sigma_{0}}^{-1}\circ W_{\sigma_{2}}^{-1}%
\circ...W_{\sigma_{l-1}}^{-1}(\rho)$ for some $\sigma_{0}\sigma_{1}%
...\sigma_{l-1}\in\{0,1\}^{l}$ for some $l\in\{1,2,...k-1\}$. We denote these
discontinuities, together with the points $0$ and $1$, by
\[
D_{k,0}:=0<D_{k,1}<D_{k,2}<....<D_{k,D(k)-1}<1=:D_{k,D(k)},
\]
where $D(1)=3,D(2)=5<D(3)<D(4)...$ . For each $k\geq1,$ one of the $D_{k,j}$'s
is equal to $\rho.$ For $k\geq1$ we have $W^{k}(x)=W_{0}^{k}(x)$ for
$x\in\lbrack D_{k,0},D_{k,1}]$ and $W_{+}^{k}(x)=W_{0}^{k}(x)$ for
$x\in\lbrack D_{k,0},D_{k,1})$. Similarly $W^{k}(x)=W_{1}^{k}(x)$ for all
$x\in(D_{k,D(k)-1},D_{k,D(k)}]$ and $W_{+}^{k}(x)=W_{1}^{k}(x)$ for
$x\in\lbrack D_{D(k)-1},D_{D(k)}]$.

For all $x\in(D_{k,l},D_{k,l+1})$ $(l=0,1,...,D(k)-1)$, $W^{k}(x)=W_{+}^{k}$
$(x)=W_{\theta_{k}}\circ W_{\theta_{k-1}}\circ...W_{\theta_{1}}(x)$ for some
fixed $\theta_{1}\theta_{2}...\theta_{k}\in\{0,1\}^{k}$. We refer to
$\theta_{1}\theta_{2}...\theta_{k}$ as the \textit{address of the interval}
$(D_{k,l},D_{k,l+1}),$ we say $(D_{k,l},D_{k,l+1})$ that "\textit{has address}
$\theta_{1}\theta_{2}...\theta_{k}$", and we write, by slight abuse of
notation, $\tau((D_{k,l},D_{k,l+1}))=\theta_{1}\theta_{2}...\theta_{k}.$

Let $k>1.$ Consider two adjacent intervals,$(D_{k,m-1},D_{k,m}]$ and
$(D_{k,m},D_{k,m+1}]$ for $m\in\{1,2,...,D(k)-1\}$ and $k>1$. Let the one on
the right have address $\theta_{0}\theta_{1}...\theta_{k-1}$ and the one on
the left have address $\eta_{0}\eta_{1}...$ $\eta_{k-1}$. Then $\eta_{0}%
\eta_{1}...$ $\eta_{k-1}\prec\theta_{0}\theta_{1}...\theta_{k-1}$ and we have%
\begin{align*}
\tau(x)|_{k-1}  &  =\eta_{0}\eta_{1}...\eta_{k-1}\text{ for all }%
x\in(D_{k,m-1},D_{k,m}]\text{,}\\
\tau^{+}(x)|_{k-1}  &  =\eta_{0}\eta_{1}...\eta_{k-1}\text{ for all }%
x\in\lbrack D_{k,m-1},D_{k,m}),\\
\tau(x)|_{k-1}  &  =\theta_{0}\theta_{1}...\theta_{k-1}\text{ for all }%
x\in(D_{k,m},D_{k,m+1}],\\
\tau^{+}(x)|_{k-1}  &  =\theta_{0}\theta_{1}...\theta_{k-1}\text{ for all
}x\in\lbrack D_{k,m},D_{k,m+1}).
\end{align*}
In particular, $\tau(x)|_{k-1}$ and $\tau^{+}(x)|_{k-1}$ are constant and
equal on each of the open intervals $(D_{k,m-1},D_{k,m})$ and have distinct
values at the discontinuity points $\{D_{k,m}\}_{m=1}^{D(k)-1}$.

\section{\label{structuresec}The structures of $\Omega$, $\Omega_{+}$ and
$\overline{\Omega}$.}

In this section we characterize $\Omega$ and $\Omega_{+}$ as certain inverse
limits, and we characterize $\overline{\Omega}$ as an attractor of an iterated
closed relation on $I^{\infty}$. These inverse limits are natural and they
clarify the structures of $\Omega$ and $\Omega_{+}.$ They are implied by the
shift invariance of $\Omega$ and $\Omega_{+}$. Recall that $S:I^{\infty
}\rightarrow I^{\infty}$ denotes the left-shift map $\sigma_{0}\sigma
_{1}\sigma_{2}...\mapsto\sigma_{1}\sigma_{2}\sigma_{3}...$.

\begin{proposition}
\label{shiftprop}(i) $\tau(W(x))=S(\tau(x))$ and $\tau^{+}(W_{+}%
(x))=S(\tau^{+}(x))$ for all $x\in\lbrack0,1]$.

(ii) $S(\Omega)=\Omega$ and $S(\Omega_{+})=\Omega_{+}.$
\end{proposition}

\begin{proof}
(i) This follows at once from the definitions of $\tau$ and $\tau^{+}.$ (ii)
This follows from (i) together with $W([0,1])=W_{+}([0,1])=[0,1].$
\end{proof}

We say that $\Lambda\subset I^{\infty}$ is \textit{closed from the left} if,
whenever $\{x_{n}\}_{n=0}^{\infty}$ is an non-decreasing sequence of points in
$\Lambda$, $\lim x_{n}\in\Lambda$. We say that $\Lambda\subset I^{\infty}$ is
\textit{closed from the right} if, whenever $\{x_{n}\}_{n=0}^{\infty}$ is
non-increasing sequence in $\Lambda$, $\lim x_{n}\in\Lambda$. For $S\subset
X,$ where $X=I^{\infty}$ or $[0,1],$ we write $L(S)=\{\sigma\in X:$
\textit{there is} \textit{a non-decreasing sequence} $\{z_{n}\}_{n=0}^{\infty
}\subset S$ \textit{with} $\sigma=\lim z_{n}\}\ $to denote the closure of $S$
from the left. Analogously, we define $R(S)$ for the closure of $S$ from the right.

\begin{proposition}
\label{closedlemma} (i) $\Omega$ is closed from the left and $\Omega_{+}$ is
closed from the right;

(ii) $\overline{\Omega}=\overline{\Omega_{+}}=\Omega\cup\Omega_{+}%
=\overline{\Omega\cap\Omega_{+}}$
\end{proposition}

\begin{proof}
Proof of (i): By (6) $\tau:[0,1]\rightarrow I^{\infty}$ is monotone strictly
increasing. By (7) $\tau$ is continuous from the left. Let $\{z_{n}%
\}_{n=0}^{\infty}$ be a non-decreasing sequence of points in $\Omega.$ Let
$y_{n}=\tau^{-1}(z_{n})$. Let $y=\lim y_{n}\in\lbrack0,1]$. Since $\tau$ is
continuous from the left, $\Omega\ni\tau(y)=\tau(\lim y_{n})=\lim\tau
(y_{n})=\lim z_{n}$. It follows that $\Omega$ is closed from the left.
Similarly, $\Omega_{+}$ is closed from the right.

Proof of (ii): Let $Q=\{x\in\lbrack0,1]:\tau(x)=\tau^{+}(x)\}$. Then by (4)
$\overline{Q}=[0,1].$ Also, by (5),
\[
\Omega\cap\Omega_{+}=\tau([0,1])\cap\tau^{+}([0,1])=\tau(Q)=\tau^{+}(Q).
\]
Hence
\[
\overline{\Omega\cap\Omega_{+}}=\overline{\tau(Q)}=\overline{\tau^{+}%
(Q)}=\overline{\Omega}=\overline{\Omega_{+}}.
\]

Finally, $\Omega\cup\Omega_{+}=L(\tau(Q))\cup R(\tau^{+}(Q))=L(\tau(Q))\cup
R(\tau(Q))=$ $\overline{\tau(Q)}=\overline{\Omega}$.
\end{proof}

We define $s_{i}:I^{\infty}\rightarrow I^{\infty}$ by $s_{i}(\sigma)=i\sigma$
$(i=0,1).$ Note that both $s_{0},$ and $s_{1}$, are contractions with
contractivity $1/2$. We write $2^{I^{\infty}}$ to denote the set of all
subsets of $I^{\infty}.$ For $\sigma,\omega\in I^{\infty}$ we define%
\begin{align*}
\lbrack\sigma,\omega]  &  :=\{\zeta\in I^{\infty}:\sigma\preceq\zeta
\preceq\omega\},\\
(\sigma,\omega)  &  :=\{\zeta\in I^{\infty}:\sigma\prec\zeta\prec\omega\},\\
(\sigma,\omega]  &  :=\{\zeta\in I^{\infty}:\sigma\prec\zeta\preceq\omega\},\\
\lbrack\sigma,\omega)  &  :=\{\zeta\in I^{\infty}:\sigma\preceq\zeta
\prec\omega\}.
\end{align*}

\begin{proposition}
\label{omegaprop}Let $\alpha=S(\tau(\rho))$ and $\beta=S(\tau^{+}(\rho)).$

(i) $\Omega=\bigcap\limits_{k\in\mathbb{N}}\Psi^{k}([\overline{0},\overline
{1}])$ where $\Psi:2^{I^{\infty}}\rightarrow2^{I^{\infty}}$ is defined by
\[
2^{I^{\infty}}\ni\Lambda\mapsto s_{0}(\Lambda\cap\lbrack\overline{0}%
,\alpha])\cup s_{1}(\Lambda\cap(\beta,\overline{1}]).
\]

(ii) $\Omega_{+}=\bigcap\limits_{k\in\mathbb{N}}\Psi_{+}^{k}([\overline
{0},\overline{1}])$ where $\Psi_{+}:2^{I^{\infty}}\rightarrow2^{I^{\infty}}$
is defined by
\[
2^{I^{\infty}}\ni\Lambda\mapsto s_{0}(\Lambda\cap\lbrack\overline{0}%
,\alpha))\cup s_{1}(\Lambda\cap\lbrack\beta,\overline{1}]).
\]

(iii) $\overline{\Omega}=\overline{\Omega_{+}}=\bigcap\limits_{k\in\mathbb{N}%
}\overline{\Psi}^{k}([\overline{0},\overline{1}])$ where $\overline{\Psi
}:2^{I^{\infty}}\rightarrow2^{I^{\infty}}$ is defined by%
\[
2^{I^{\infty}}\ni\Lambda\mapsto s_{0}(\Lambda\cap\lbrack\overline{0}%
,\alpha])\cup s_{1}(\Lambda\cap\lbrack\beta,\overline{1}]).
\]

\end{proposition}

\begin{proof}
Proof of (i): Let $S|_{\Omega}:\Omega\rightarrow\Omega$ denote the domain and
range restricted shift map. It is readily found that the branches of
$S|_{\Omega}^{-1}:\Omega\rightarrow\Omega$ are $s_{0}|_{\Omega}:[\overline
{0},\alpha]\cap\Omega\rightarrow\Omega$ where
\[
s_{0}|_{\Omega}(\sigma)=s_{0}(\sigma)=0\sigma\text{ for all }\sigma\in
\lbrack\overline{0},\alpha]\cap\Omega,
\]
and $s_{1}|_{\Omega}:(\beta,\overline{1}]\cap\Omega\rightarrow\Omega$ where
\[
s_{1}|_{\Omega}(\sigma)=s_{1}(\sigma)=1\sigma\text{ for all }\sigma\in
(\beta,\overline{1}]\cap\Omega.
\]
(Note that $\alpha_{0}=1,$ $\beta_{0}=0$ and $\beta\prec\alpha.)$ It follows
that
\[
S|_{\Omega}^{-1}(\Lambda)=s_{0}(\Lambda\cap\lbrack\overline{0},\alpha])\cup
s_{1}(\Lambda\cap(\beta,\overline{1}])=\Psi(\Lambda)
\]
for all $\Lambda\subset\Omega$. Since $\Omega\subset\lbrack\overline
{0},\overline{1}]$ it follows that
\[
\Omega=S|_{\Omega}^{-1}(\Omega)=\Psi(\Omega)\subset\Psi([\overline
{0},\overline{1}]).
\]
Also, since $\Psi([\overline{0},\overline{1}])\subset\lbrack\overline
{0},\overline{1}]$ it follows that $\{\Psi^{k}([\overline{0},\overline{1}])\}$
is a decreasing (nested) sequence of sets, each of which contains $\Omega;$
hence%
\[
\Omega\subset\bigcap\limits_{k\in\mathbb{N}}\Psi^{k}([\overline{0}%
,\overline{1}])\text{.}%
\]

It remains to prove that $\Omega\supset\bigcap\limits_{k\in\mathbb{N}}\Psi
^{k}([\overline{0},\overline{1}])$. We note that $s_{0}([\overline{0}%
,\alpha])=[\overline{0},\tau(\rho)]$ and $s_{1}((\beta,\overline{1}%
])=(\tau^{+}(\rho),\overline{1}],$ from which it follows that
\begin{equation}
\bigcap\limits_{k\in\mathbb{N}}\Psi^{k}([\overline{0},\overline{1}%
])=\bigcap\limits_{k\in\mathbb{N}}\{\sigma\in I^{\infty}:S^{k}(\sigma
)\in\lbrack\overline{0},\tau(\rho)]\cup(\tau^{+}(\rho),\overline{1}]\}.
\label{nameeq}%
\end{equation}
Let $\omega\in\bigcap\limits_{k\in\mathbb{N}}\Psi^{k}([\overline{0}%
,\overline{1}])$. Suppose $\omega\notin\Omega$. Let
\[
\omega_{-}=\sup\{\sigma\in\Omega:\sigma\preceq\omega\}\text{ and }\omega
_{+}=\inf\{\sigma\in\Omega:\omega\preceq\sigma\},
\]
so that%
\[
\omega_{-}\preceq\omega\preceq\omega_{+}\text{.}%
\]
But $\omega_{-}\in\Omega$ (since $\Omega$ is closed from the left), so%
\[
\omega_{-}\prec\omega\preceq\omega_{+}\text{.}%
\]
Note that, since $\inf\{\sigma\in\Omega:\omega\preceq\sigma\}$ $=\inf
\{\sigma\in\Omega_{+}:\omega\preceq\sigma\}$, and $\Omega_{+}$ is closed from
the right, we have $\omega_{+}\in\Omega_{+}$. Let $K=\min\{k\in\mathbb{N}%
:\left(  \omega_{-}\right)  _{k}\neq\left(  \omega_{+}\right)  _{k}\}$. Then
$S^{K}(\omega_{-})\prec S^{K}(\omega)\preceq S^{K}(\omega_{+})$ and we must
have $S^{K}(\omega_{-})=\tau(\rho)$ and $S^{K}(\omega_{+})=\tau^{+}(\rho)$. So%
\[
\tau(\rho)\prec S^{K}(\omega)\preceq\tau^{+}(\rho)\text{,}%
\]
therefore $\omega\notin\{\sigma\in I^{\infty}:S^{K}(\sigma)\in\lbrack
\overline{0},\tau(\rho)]\cup(\tau^{+}(\rho),\overline{1}]\}$ which, because of
(\ref{nameeq}), contradicts our assumption that $\omega\in\bigcap
\limits_{k\in\mathbb{N}}\Psi^{k}([\overline{0},\overline{1}])$. Hence
$\omega\in\Omega$ and we have
\[
\Omega\supset\bigcap\limits_{k\in\mathbb{N}}\Psi^{k}([\overline{0}%
,\overline{1}])\text{.}%
\]

This completes the proof of (i).

Proof of (ii): similar to the proof of (i), with the role of $[\overline
{0},\tau(\rho)]$ played by $[\overline{0},\tau(\rho))$ and the role of
$(\tau^{+}(\rho),\overline{1}]$ played by $[\tau^{+}(\rho),\overline{1}]$.

Proof of (iii): similar to the proofs of (i) and (ii).
\end{proof}

It is helpful to note that the addresses $\alpha$ and $\beta$ in Proposition
\ref{omegaprop} obey%
\begin{align*}
\alpha &  =\tau(W_{0}(\rho)),\beta=\tau(W_{1}(\rho)),\\
\tau(\rho)  &  =0\alpha=01\alpha_{1}\alpha_{2}...\text{ and }\tau^{+}%
(\rho)=0\beta=10\beta_{1}\beta_{2}....
\end{align*}
Let $M>0$ be such that $D_{k,M+1}=\rho$. It follows from the discussion at the
end of Section \ref{basicsec} that $\tau((D_{k,M},\rho))=\tau^{+}%
((D_{k,M},\rho))=01\alpha_{1}\alpha_{2}..\alpha_{k-2}$ and $\tau
((\rho,D_{k,M+2}))=\tau^{+}((\rho,D_{k,M+2}))=10\beta_{1}\beta_{2}%
...\beta_{k-2}$.

\begin{corollary}
\label{addresscor}Let $k\geq1,$ $\alpha=\tau(W_{0}(\rho))$, $\beta=\tau
(W_{1}(\rho)),$ and let $M>0$ be such that $D_{k,M+1}=\rho$. The set of
addresses $\{\tau((D_{k,l},D_{k,l+1}))\}_{l=0}^{D(k)-1}$ is uniquely
determined by $\alpha|_{k-1}$ and $\beta|_{k-1}$. For some $n_{1},n_{2}$ such
that $0\leq n_{1}<M<n_{2}\leq D(k)-1$, $\tau((D_{k,n_{1}},D_{k,n_{1}%
+1}))=\beta_{0}\beta_{1}...\beta_{k-2}\beta_{k-1}$ and $\tau((D_{k,n_{2}%
},D_{k,n_{2}+1}))=\alpha_{0}\alpha_{1}...\alpha_{k-2}\alpha_{k-1}$. The set of
addresses $\{\tau((D_{k,l},D_{k,l+1})):l\in\{0,1,...,D(k)-1\},l\neq
n_{1},l\neq n\}$ are uniquely determined by $\alpha|_{k-2}$ and $\beta|_{k-2}%
$; for example, $\tau((D_{k,M},\rho))=$ $0\alpha_{0}\alpha_{1}...\alpha_{k-2}%
$, and $\tau((\rho,D_{k,M+2}))=1\beta_{0}\beta_{1}...\beta_{k-2}$.

\begin{proof}
It follows from Proposition \ref{omegaprop} that the set of addresses at level
$k$, namely $\{\tau((D_{k,l},D_{k,l+1}))\}_{l=0}^{D(k)-1},$ is invariant under
the following operation: put a $"0"$ in front of each address that is less
than or equal to $\alpha$, then truncate back to length $k;$ take the union of
the resulting set of addresses with the set of addresses obtained by: put a
$"1"$ in front of each address that is greater than or equal to $\beta$, and
drop the last digit.
\end{proof}
\end{corollary}

\section{\label{symmetrysec}Symmetry of $\overline{\Omega}$ and a consequent
homeomorphism of $[0,1]$}

\begin{lemma}
\label{lemma1} $\overline{\Omega}=\{\sigma\in I^{\infty}:\text{ for all }%
k\in\mathbb{N}$, $\sigma_{k}=0\Rightarrow S^{k}(\sigma)\preceq\tau(\rho)$ and
$\sigma_{0}=1\Rightarrow\tau^{+}(\rho)\preceq S^{k}(\sigma)\}.$
\end{lemma}

\begin{proof}
This is an immediate consequence of Proposition \ref{omegaprop}.
\end{proof}

\begin{corollary}
\label{corollary2} $\overline{\Omega}$ is symmetric if and only if
$\alpha=\beta^{\ast}$ (or equivalently $\tau(\rho)=\left(  \tau^{+}%
(\rho)\right)  ^{\ast}$).
\end{corollary}

\begin{lemma}
\label{lemma3} The maps $\tau(\rho)$ and $\tau^{+}(\rho)$ are strictly
increasing as functions of $\rho\in\lbrack a,b]$ to $I^{\infty}$.
\end{lemma}

\begin{proof}
Note that $\tau(\rho)$ depends both implicitly and explicitly on $\rho.$ Let
$1-b\leq\rho<\rho^{\prime}\leq a$ be such that $\tau(\rho)\succeq$ $\tau
(\rho^{\prime})$. Observe that $\tau(\rho)|_{0}=$ $\tau(\rho^{\prime})|_{0}$.

Assume first that there is a largest $n>0$ such that $\tau(\rho)|_{n}=$
$\tau(\rho^{\prime})|_{n}:=\theta_{0}\theta_{1}...\theta_{n}$. Then $\tau
(\rho)=\theta_{0}\theta_{1}...\theta_{n}1...$ and $\tau^{+}(\rho)=\theta
_{0}\theta_{1}...\theta_{n}0...,$ which implies
\begin{equation}
W_{\rho}^{n}(\rho)\geq\rho\text{ and }W_{\rho^{\prime}}^{n}(\rho^{\prime}%
)\leq\rho^{\prime}\text{.} \label{inequalityeq}%
\end{equation}
(We write $W=W_{\rho}$ when we want to note the dependence on $\rho$. ) We may
assume that $\tau(\rho)|_{n}$ is constant on $[\rho,\rho^{\prime}]$ for
otherwise we can restrict to a smaller interval with a strictly smaller value
of $n$. As a consequence, at every iteration, we apply the same branch $W_{0}$
or $W_{1}$ to $W_{\xi}$ to compute compute $g(\xi):=W_{\xi}^{n}(\xi)$ for all
$\xi\in\lbrack\rho,\rho^{\prime}]$. Therefore $g$ is continuous with
derivative at least $d^{n}>1,$ which contradicts (\ref{inequalityeq})$.$

The only remaining possibility is that $\tau(\rho)=\tau(\rho^{\prime})$. We
may assume that $\tau(\rho)$ is constant on $[\rho,\rho^{\prime}],$ otherwise
we can reduce the problem to the previous case. This would mean that for
arbitrarily large $n,$ the image of the interval $[\rho,\rho^{\prime}]$ under
$g$ is at an interval of size least $d^{n}(\rho^{\prime}-\rho),$ a contradiction.

Essentially the same argument, with the role of $\tau$ played by $\tau^{+}$
and the role of\ played by $W_{+}$, proves that $\tau^{+}(\rho)$ is strictly
increasing as a function of $\rho\in[1-b,a]$ to $I^{\infty}.$
\end{proof}

\begin{corollary}
\label{corollary4} The map $\rho\mapsto\tau(\rho)$ is left continuous and the
map $\rho\mapsto\tau^{+}(\rho)$ is right continuous.
\end{corollary}

\begin{proof}
Fix a parameter $\rho_{0}$ and let $\varepsilon>0$. Then by (7) there is
$x<\rho_{0}$ which is not a preimage of $\rho_{0}$ for any order and such
that
\[
d(\tau_{\rho_{0}}^{+}(x),\tau_{\rho_{0}}(\rho_{0}))<\frac{\varepsilon}{2}.
\]
By (9), for any $n\in\mathbb{N}$ there exists $\delta>0$ such that the prefix
$\tau_{\rho}^{+}(x)|_{n}$ is constant when $\rho\in(\rho_{0}-\delta,\rho
_{0}+\delta)$. Let $n$ be such that $2^{-n}<\varepsilon$, and let $\rho>x$ and
$\rho\in(\rho_{0}-\delta,\rho_{0})$. We have that $\tau_{\rho}^{+}(x)\prec
\tau_{\rho}^{+}(\rho)$ and
\[
d(\tau_{\rho}^{+}(x),\tau_{\rho_{0}}^{+}(x))<\frac{\varepsilon}{2}.
\]
Combining the two inequalities we obtain%
\[
d(\tau_{\rho}^{+}(x),\tau_{\rho_{0}}(\rho_{0}))<\varepsilon,
\]
and by Lemma \ref{lemma3} we also have%
\[
\tau_{\rho}^{+}(x)\prec\tau_{\rho}(\rho)\prec\tau_{\rho_{0}}(\rho_{0}).
\]
The distance $d$ has the property that if $\sigma\prec\zeta\prec\sigma
^{\prime}$ then $d(\sigma,\zeta)\leq d(\sigma,\sigma^{\prime})$ and
$d(\zeta,\sigma^{\prime})\leq d(\sigma,\sigma^{\prime})$. This shows that
$\rho\mapsto\tau(\rho)$ is left continuous. The right continuity of
$\rho\mapsto\tau^{+}(\rho)$ admits an analagous proof.
\end{proof}

As a consequence of Corollary \ref{corollary2}, Lemma \ref{lemma3} and (10),
we obtain the unicity of $\rho$ for which $\overline{\Omega}$ is symmetric.As
a consequence of Corollary \ref{corollary2}, Lemma \ref{lemma3} and (10), we
obtain the unicity of $\rho$ for which $\overline{\Omega}$ is symmetric.

\begin{corollary}
\label{corollary5} There is at most one $\rho\in\lbrack1-b,a]$ such that
$\overline{\Omega}=\overline{\Omega}^{\ast}$.
\end{corollary}

\begin{proof}
[Proof of Theorem \ref{symmetrythm}.]By Lemma \ref{lemma3} and (10), we may
define%
\[
\rho_{0}:=\sup\{\rho\in\lbrack1-b,a]:\tau(\rho)\preceq\tau^{+}(\rho)^{\ast
}\}=\inf\{\rho\in\lbrack1-b,a]:\tau(\rho)^{\ast}\preceq\tau^{+}(\rho)\}.
\]
Assume $\tau(\rho_{0})\prec\tau^{+}(\rho_{0})^{\ast}$. It is straighfoward to
check $1-b<\rho_{0}<a$.

There is a largest $n\geq2$ such that $\tau(\rho_{0})|_{n}=\tau^{+}(\rho
_{0})^{\ast}|_{n}=:\eta=01...$ .

Observe that $\tau(\rho_{0})=0\tau(W_{0}(\rho_{0}))$ and $\tau^{+}(\rho
_{0})=1\tau^{+}(W_{1}(\rho_{0}))$. If neither $W_{0}(\rho_{0})$ nor
$W_{1}(\rho_{0})$ belongs to $\{0,1\}\cup\bigcup\limits_{k=0}^{n-1}W^{-k}%
(\rho_{0}),$ then by (9) both $\tau(\rho)|_{n+1}$ and $\tau^{+}(\rho)|_{n+1}$
are constant on a neighborhood of $\rho_{0}$ which contradicts the definition
of $\rho_{0}$.

Let us consider the projection $\hat{\pi}(\tau^{+}(W_{1}(\rho_{0}))^{\ast})$.
If $\hat{\pi}(\tau^{+}(W_{1}(\rho_{0}))^{\ast})>W_{0}(\rho_{0})$ then by the
continuity of $W_{0}$, of $\hat{\pi}$ by (13), of $\rho\mapsto\tau_{\rho}%
^{+}(\rho)$ (Corollary \ref{corollary4}) there is a $\rho>\rho_{0}$ such that
$\hat{\pi}_{\rho}(\tau_{\rho}^{+}(W_{1}(\rho_{0}))^{\ast})>W_{0}(\rho)$. By
(6) and (13) this implies $\tau_{\rho}(\rho)\prec\tau_{\rho}^{+}(\rho)^{\ast}%
$, which again contradicts the definition of $\rho_{0}$.

As $\hat{\pi}$ is increasing, (13) and $\tau(W_{0}(\rho_{0}))\prec\tau
^{+}(W_{1}(\rho_{0}))^{\ast},$ we have $\hat{\pi}(\tau^{+}(W_{1}(\rho
_{0}))^{\ast})=W_{0}(\rho_{0})$. Let $0<m<n$ be minimal such that $W^{m}\circ
W_{0}($ $\rho_{0})=\rho_{0}$ or $W^{m}\circ W_{1}(\rho_{0})=\rho_{0}$. We may
apply (14) $m$ times and obtain%
\begin{equation}
W^{m}\circ W_{0}(\rho_{0})=\hat{\pi}(S^{m}(\tau^{+}(W_{1}(\rho_{0}))^{\ast
}))=\hat{\pi}(\tau^{+}(W^{m}\circ W_{1}(\rho_{0}))^{\ast})\text{.}
\label{equation2}%
\end{equation}
As $\tau^{+}(\rho_{0})=1...$, if $W^{m}\circ W_{1}(\rho_{0})=\rho_{0}$ then we
have%
\[
\tau(\rho_{0})\prec\tau^{+}(\rho_{0})^{\ast}=\tau^{+}(W^{m}\circ W_{1}%
(\rho_{0}))^{\ast}\prec\tau^{+}(\rho_{0}),
\]
which by (6) and equation (\ref{equation2}) implies $W^{m}\circ W_{0}(\rho
_{0})=\rho_{0}$. Therefore $\tau(\rho_{0})=\tau^{+}(\rho_{0})^{\ast}$ as both
are periodic of period $m+1$ and have the same prefix of length $n>m$, a contradiction.

If $W^{m}\circ W_{1}(\rho_{0})\neq\rho_{0}$ then $W^{m}\circ W_{0}(\rho
_{0})=\rho_{0}$ thus by (13), (10) and equality (\ref{equation2}) we obtain%
\[
\tau(\rho_{0})\prec\tau^{+}(\rho_{0})^{\ast}\preceq\tau^{+}(W^{m}\circ
W_{1}(\rho_{0})):=\sigma^{\prime}.
\]
By (6), this means that $\rho_{0}\leq W^{m}\circ W_{1}(\rho_{0})$ so in fact%
\[
\rho_{0} < W^{m}\circ W_{1}(\rho_{0})\text{.}%
\]

As $W^{m+1}(\rho_{0})=\rho_{0},$ $\tau(\rho_{0})=$ $\kappa\kappa
\kappa...:=\kappa^{\infty}$ where $\kappa=\tau(\rho_{0})|_{m+1}=\tau^{+}
(\rho_{0})^{*}|_{m+1}$, as $m+1 \leq n$. We can write $\tau^{+}(\rho_{0})
=\kappa^{\ast}\sigma^{\prime}$ therefore $\kappa^{\ast}\sigma^{\prime}%
\prec\sigma^{\prime}$ by (6) and the previous inequality. By induction we get
$\kappa^{\ast\infty}\prec$ $\sigma^{\prime}$ so
\[
\tau^{+}(\rho_{0})^{\ast}=\kappa(\sigma^{\prime\ast})\prec\kappa^{\infty}%
=\tau(\rho_{0})\text{,}%
\]
a contradiction.

The case $\tau(\rho_{0})\succ\tau^{+}(\rho_{0})^{\ast}$ is analagous by the
symmetric definition of $\rho_{0}$, therefore $\overline{\Omega}_{\rho_{0}}$
is symmetric.
\end{proof}

\begin{proposition}
\label{proposition7} If $\overline{\Omega}=\overline{\Omega}^{\ast}$then the
map $h:[0,1]\rightarrow\lbrack0,1]$ defined by $h(x)=\hat{\pi}(\tau(x)^{*})$
is a homeomorphism and $h\circ\hat{\pi}=\hat{\pi}\circ^{\ast}$ on $I^{\infty}$.
\end{proposition}

\begin{proof}
First by Corollary \ref{corollary2}, we have $\tau(\rho)=\tau^{+}(\rho)^{\ast
}$ and points $x$ for which $\tau(x)\neq\tau^{+}(x)$ are exactly preimages of
$\rho$. In this case, there is $n\geq0$ such that $\tau(x)$ and $\tau^{+}(x)$
have the same initial prefix $\kappa:=\tau(x)|_{n}=\tau^{+}(x)|_{n},$ and
$\tau(x)=$ $\kappa\tau(\rho),\tau^{+}(x)=$ $\kappa\tau^{+}(\rho)$. Therefore,
by (13), for all $x\in\lbrack0,1],$ we have
\[
\tau(h(x))=\tau^{+}(x)^{\ast}\text{ and }\tau^{+}(h(x))=\tau(x)^{\ast}\text{,}%
\]
thus $h\circ h(x)=x$. By (6), (10) and (13), $h$ is also decreasing. Therefore
$h:[0,1]\rightarrow$ $[0,1]$ is a homeomorphism.

It is readily verified, by direct substitution of $x=h\circ\hat{\pi}(\sigma)$
into $h(x)=\hat{\pi}(\tau(x)^{\ast})$, that $h\circ\hat{\pi}\left(
\sigma\right)  =\hat{\pi}(\sigma^{\ast})$ for all $\sigma\in I^{\infty}$.
\end{proof}

\section{Iterated Closed Relations and Conley Decomposition for Itineraries of
$W$}

Theorem \ref{irtheorem} follows from Proposition \ref{omegaprop}, but some
extra language is needed. In explaining this language we describe the
Conley-McGehee-Wiandt decomposition theorem, \cite[Theorem 13.1]{mcgehee}.

For $X$ a compact Hausdorff space, let $2^{X}$ be the subsets of $X$. A
\textit{relation} $r$ on $X$ is simply a subset of $X\times X.$ A relation $r$
on $X$ is called a \textit{closed relation} if $r$ is a closed of $X\times X.$
For example the set $r\subset I^{\infty}\times I^{\infty}$ defined in Theorem
\ref{irtheorem}, namely
\[
r=\{(0\sigma,\sigma)\in I^{\infty}\times I^{\infty}:\sigma\preceq\alpha
\}\cup\{(1\sigma,\sigma)\in I^{\infty}\times I^{\infty}:\beta\preceq\sigma\},
\]
is a closed relation. Following \cite{mcgehee}, a relation $r\in2^{X}$
provides a mapping $r:2^{X}\rightarrow2^{X}$ defined by%
\[
r(C)=\{y\in X:(x,y)\in r\text{ for some }x\in C\}\text{.}%
\]
Notice that the image of a nonempty set may be empty. Iterated relations are
defined by $r^{0}=X\times X$ and, for all $k\in\mathbb{N}$,%
\[
r^{k+1}=r\circ r^{k}=\{(x,z):(x,y)\in r,(y,z)\in r^{k}\text{ \textit{for some
}}y\in X\}.
\]

The \textit{omega limit set} of $C\subset X$ under a closed relation $r\subset
X\times X$ is%
\[
\omega(C)=\cap\mathfrak{K}(C)
\]
where
\[
\mathfrak{K}(C)=\{D\text{ \textit{is a closed subset of }}X:r(D)\cup
r^{n}(C)\subset D\text{ \textit{for some} }n\in\mathbb{N}\}\text{.}%
\]
By definition, an \textit{attractor} of a closed relation $r$ is a closed set
$A$ such that the following two conditions hold:

(i) $r(A)=A$;

(ii) there is a closed neighborhood $\overline{\mathcal{N}}(A)$ of $A$ such
that $\omega(C)\subset A$ for all $C\subset\overline{\mathcal{N}}(A)$.

The basin $\mathcal{B}(A)$ of an attractor $A$ for a closed relation $r$ on a
compact Hausdorff space $X$ is the union of all open sets $O\subset X$ such
that $\omega(C)\subset A$ for all $C\subset O.$

Given an attractor $A$ for a closed relation $r$ on a compact Hausdorff space
$X$, there exists a corresponding \textit{attractor block}, namely a closed
set $E\subset X$ such that $E$ contains both $A$ and $r(E)$ in its interior,
and $A=\omega(E)$. Also, there exists a unique \textit{dual repeller}
$A^{\ast}=X\backslash\mathcal{B}(A).$ This repeller is an attractor for the
transpose relation $r^{\ast}=\{(y,x):(x,y)\in r\}$. The set of connecting
orbits associated with the attractor/repeller pair $A,$ $A^{\ast}$ is
$\mathcal{C}(A)=X\backslash(A\cup A^{\ast})$.

If $r$ is a closed relation on a compact Hausdorff space $X$, then $x\in X$ is
called \textit{chain-recurrent} for $r$ if for every closed neighborhood $f$
of $r$, $x$ is periodic for $f$ (i.e. there exists a finite sequence of points
$\{x_{n}\}_{n=0}^{p-1}\subset X$ such that $x_{0}=x,$ $(x_{p-1},x_{0})\in f$
and $(x_{n-1},x_{n})\in f$ for $n=1,2,...,p-1$). The chain recurrent set
$\mathcal{R}$ for $r$ is the union of all the points that are chain recurrent
for $r$. A transitive component of $\mathcal{R}$ is a member of the
equivalence class on $\mathcal{R}$ defined by $x\sim y$ when for every closed
neighborhood $f$ of $r$ there is an orbit from $x$ to $y$ under $f$ (i.e.
there exists a finite sequence of points $\{x_{n}\}_{n=0}^{p-1}\subset
\mathcal{R}$ such that $x_{0}=x,$ $x_{p-1}=y,$ and $(x_{n},x_{n+1})\in f$ for
all $n\in\{0,1,...p-1\}$.)

\begin{theorem}
[Conley-McGehee-Wiandt]If $r$ is a closed relation on a compact Hausdorff
space $X,$ then
\[
\mathcal{R}=\bigcup\limits_{A\in\mathcal{U}}\mathcal{C}(A)
\]
where $\mathcal{R}$ is the chain-recurrent set and $\mathcal{U}$ is the set of attractors.
\end{theorem}

\begin{proof}
[\textbf{Proof of Theorem \ref{irtheorem}}]This follows at once from
Proposition \ref{omegaprop} together with the definitions associated with the
theory of iterated closed relations as summarized above, but see
\cite{mcgehee}.
\end{proof}

We note the following. $\overline{\Omega}$ can be embedded in $[0,1]\subset
\mathbb{R}$ using the (continuous and surjective) coding map $\pi:I^{\infty
}\rightarrow\lbrack0,1]$ associated with the iterated function system
$([0,1];x\mapsto x/2,x\mapsto(1+x)/2).$ This coding map $\pi$ is defined for
all $\sigma$%
\[
\pi(\sigma)=\sum\limits_{k\in\mathbb{N}}\frac{\sigma_{k}}{2^{k+1}}\text{.}%
\]
$\pi$ provides a homeomorphism between $\overline{\Omega}$ and $\pi
(\overline{\Omega})$. The point $\sigma\in\overline{\Omega}$ is uniquely and
unambiguously represented by the binary real number $0.\sigma$. In the
representation provided by $\pi,$ the map $\overline{\Psi}:2^{I^{\infty}%
}\rightarrow2^{I^{\infty}}$ becomes the action of the iterated closed relation
$\widetilde{r}\subset\lbrack0,1]\times\lbrack0,1]\subset\mathbb{R}^{2}$
defined by%
\[
\widetilde{r}:=\{(x,x/2):x\in\lbrack0,\pi(\alpha)]\}\cup\{(x,(x+1)/2):x\in
\lbrack\pi(\beta),1]\}
\]
on subsets of $[0,1]$. It follows from Proposition \ref{omegaprop} (iii)\ that
$\pi(\overline{\Omega})$ is the maximal attractor, as defined in
\cite{mcgehee}, of $\widetilde{r}$. The corresponding dual repeller is the
empty set. The chain recurrent set of $\widetilde{r}$ is $\{0,1\}\cup
(\pi(\overline{\Omega})\cap(\pi(\beta),\pi(\alpha)))$, the transitive
components are $\{0\},$ $\{1\}$ and $\pi(\overline{\Omega})\cap(\pi(\beta
),\pi(\alpha))$, and the points $\{0\}$ and $\{1\}$ are repulsive fixed points
and comprise the boundary of the basin of the attractor $\pi(\overline{\Omega
})\cap(\pi(\beta),\pi(\alpha)),$ which contains no invariant subsets.

\end{document}